\documentclass{amsart}

\usepackage{amssymb}
\usepackage{amsmath}
\usepackage{amscd}
\usepackage{graphicx}
\usepackage{latexsym}

\newcommand{\CP}{\mathbb{CP}}

\def\Z{\mathbb{Z}}

\newtheorem{theorem}{Theorem}[section]

\newtheorem{proposition}[theorem]{Proposition}
\newtheorem{corollary}[theorem]{Corollary}

\newtheorem{remark}[theorem]{Remark}
\newtheorem{example}[theorem]{Example}

\def \x {\times}
\def \eu{{\text{e}}}

\begin{document}
\title[Existence of broken Lefschetz fibrations]
{Existence of broken Lefschetz fibrations}
\author[R. \.Inan\c{c} Baykur] {R.\, \.Inan\c{c} Baykur}
\address{Max-Planck-Institut f\"ur Mathematik, Vivatgasse 7, Bonn D-53111, Germany}
\email{baykur@mpim-bonn.mpg.de, baykur@math.columbia.edu}

\begin{abstract}
We prove that every closed oriented smooth $4$-manifold $X$ admits a broken Lefschetz fibration (aka singular Lefschetz fibration) over the $2$-sphere. Given any closed orientable surface $F$ of square zero in $X$, we can choose the fibration so that $F$ is a fiber. Moreover, we can arrange it so that there is only one Lefschetz critical point when the Euler characteristic $\eu(X)$ is odd, and none when $\eu(X)$ is even. We make use of topological modifications of smooth maps with fold and cusp singularities due to Saeki and Levine, and thus we get alternative proofs of previous existence results. Also shown is the existence of broken Lefschetz pencils with connected fibers on any near-symplectic $4$-manifold.
\end{abstract}
\maketitle

\section{Introduction}

Broken Lefschetz fibrations and pencils were introduced by Denis Auroux, Simon Donaldson and Ludmil Katzarkov in \cite{ADK} under the name ``singular Lefschetz fibrations and pencils'', where the authors extended the earlier results of Donaldson on Lefschetz pencils and symplectic structures to the near-symplectic setting. Using approximately holomorphic techniques, they proved the existence of broken Lefschetz pencils on any closed oriented $4$-manifold with $b^+>0$, the point being that these admit near-symplectic structures. After blowing-up these manifolds at the base locus of the pencil, one gets a broken Lefschetz fibration over the $2$-sphere. In the same article, the authors gave an example of a broken Lefschetz fibration on a $4$-manifold which is not near-symplectic, the $4$-sphere, provoking the curiosity about which $4$-manifolds support broken Lefschetz fibrations \cite{GK, P}.

The existence of broken Lefschetz fibrations on arbitrary $4$-manifolds was later studied by David Gay and Rob Kirby by means of handlebody manipulations \cite{GK}. They constructed achiral broken Lefschetz fibrations on $2$-handlebodies that make up the given $4$-manifold, which could be assembled along the open books on their boundaries using Eliashberg's classification of overtwisted contact structures and Giroux's correspondence between open books and contact structures. ``Achirality'', i.e. allowing Lefschetz critical points with nonstandard orientations, was ultimately needed to match the open books. This way the authors produced \textit{achiral} broken Lefschetz fibrations from arbitrary closed oriented $4$-manifolds to the $2$-sphere. Moreover, they conjectured that achirality was essential in this statement, and in particular that $\CP^2$ would not admit a broken Lefschetz fibration. As we will discuss below, this conjecture turns out to be false. (Also see Remark \ref{ImmersedRoundImage} and Example \ref{BLFexample}.)

To establish the existence of broken Lefschetz fibrations on arbitrary closed oriented $4$-manifolds, we adopt a rather elementary approach. We translate the problem into the classical theory of singularities of smooth maps so to employ the methods \pagebreak of elimination of singularities due to Saeki \cite{S} and Levine \cite{Lev}. The reduction of the number of Lefschetz critical points is an extra feature that we obtain as a result of this approach. Our first theorem is:

\begin{theorem} \label{BLFexistence}
Let $X$ be a closed oriented smooth $4$-manifold and $F$ be a closed orientable surface of square zero in $X$. Then there exists a broken Lefschetz fibration from $X$ to the $2$-sphere where $F$ is a fiber. Moreover, the fibration can be chosen so that there is only one Lefschetz critical point if $\eu(X)$ is odd and none if $\eu(X)$ is even.
\end{theorem}

Another problem that remained unanswered was to establish the existence of broken Lefschetz pencils on manifolds with $b^+ >0$ without appealing to the approximately holomorphic theory. Note that achirality was invoked in Gay and Kirby's attempt to recapture this fact as well. A refinement of our arguments for the above theorem will allow us to provide a topological proof of:

\begin{theorem} \label{BLPexistence}
Let $X$ be a closed oriented smooth $4$-manifold containing a closed orientable surface $F$ of positive square. Then $X$ admits a broken Lefschetz pencil where $F$ is a fiber. Moreover, the pencil can be chosen so that there is only one Lefschetz critical point if $\eu(X)$ is odd and none if $\eu(X)$ is even.
\end{theorem}

\noindent Recall that in a pencil the base locus is by definition non-empty, and that the fiber of a pencil is  defined by adjoining the base points to the preimage of a regular value. The $b^+>0$ condition is equivalent to the existence of a closed orientable surface of positive square in the $4$-manifold, which appears as the fiber of the pencil that is Poincar\'e dual of a high multiple of the near-symplectic form in Auroux, Donaldson and Katzarkov construction. Thus the first part of the above theorem is a reformulation of the existence result in \cite{ADK}.

A broken Lefschetz fibration on a connected $4$-manifold may have disconnected fibers; examples can be found in \cite{B, P}. An open problem posted by Perutz was which closed $4$-manifolds support near-symplectic broken Lefschetz pencils with embedded singular image and connected fibers \cite{P}. We will address this problem with the following proposition, whose proof demonstrates an advantage of broken Lefschetz fibrations with immersed singular images.

\begin{proposition} \label{connected}
Let $f_0 : X \to S^2$ be a broken Lefschetz fibration with disconnected fibers whose round locus is connected and has embedded image. Then there exists a homotopy $f_t: X \to S^2$, $t \in [0,1]$, where all but one $f_t$ are broken Lefschetz fibrations, so that $f_1: X \to S^2$ is a broken Lefschetz fibration with connected fibers whose round locus is connected and has embedded image.
\end{proposition}

\noindent It is known through the results on reducing the number of near-symplectic circles \cite{P0, Ta} and the main theorem of \cite{ADK} that one can obtain a near-symplectic pencil on any closed $4$-manifold with $b^+>0$ whose round locus is connected and has embedded image. So the above proposition implies that:

\begin{corollary} \label{Always}
Any closed oriented smooth $4$-manifold with $b^+>0$ admits a near-symplectic broken Lefschetz pencil with connected round locus, embedded round image, and connected fibers.
\end{corollary}

Definitions and essential facts regarding the singularities we consider in our article are given in the next section, Theorems \ref{BLFexistence} and \ref{BLPexistence} are proved in Section \ref{Sec:Existence}, and proofs of Proposition \ref{connected} and Corollary \ref{Always} along with a short digression on isotopies of round singularities are contained in Section \ref{Sec:Connected}.

\medskip
\section{Lefschetz, fold and cusp singularities}

Let $X$ and $\Sigma$ be compact oriented manifolds of dimension four and two, respectively, and $f\colon\, X\to \Sigma$ be a smooth map. The map $f$ is said to have a \emph{Lefschetz singularity} at a point $x \in Int(X)$, if around $x$ and $f(x)$ one can choose orientation preserving charts so that $f$ conforms the complex local model
\[(u, v) \mapsto u^2 + v^2 .\]
The map $f$ is said to have a \emph{round singularity} (aka \emph{indefinite quadratic singularity}) along a $1$-manifold $Z \subset X$ if around every $z \in Z$, there are coordinates $(t, x_1, x_2, x_3)$ with $t$ a local coordinate on $Z$, in terms of which $f$ is given by
\[(t, x_1, x_2, x_3) \mapsto (t, x_1^2 + x_2^2 - x_3^2).\]
A \emph{broken Lefschetz fibration} is then defined as a smooth surjective map $f\colon\, X\to \Sigma$ which is submersion everywhere except for a finite set of points $C$ and a finite collection of circles $Z \subset X \setminus C$, where it has Lefschetz singularities and round singularities, respectively \cite{ADK}. We call the $1$-manifold $Z$ the \emph{round locus} and its image $f(Z)$ the \emph{round image} of $f$. A \emph{broken Lefschetz pencil} is defined similarly when $\Sigma = S^2$, by assuming that there is also a nonempty finite set of points $B$ in $X \setminus (Z \cup C)$ where $f$ conforms the complex local model
\[(u, v) \mapsto u / v, \]
and such that $f: X \setminus B \to S^2$ is a broken Lefschetz fibration. The points in $B$ are called \emph{base points} of the pencil. Generalizing the Thurston-Gompf construction, Auroux, Donaldson and Katzarkov proved that if there exists a $2$-dimensional real cohomology class on $X$ evaluating positively on all fiber components of a broken Lefschetz fibration (resp. pencil), then $X$ can be equipped with a deformation class of near-symplectic forms so that all the fiber components of the fibration (resp. pencil) are symplectic \cite{ADK}. Whenever this cohomological conditions is satisfied, we will call the broken Lefschetz fibration (resp. pencil) \emph{near-symplectic}. Lastly, an \emph{achiral} broken Lefschetz fibration or a pencil is defined by allowing the local models around Lefschetz singular points and the base points to be given by orientation reversing charts.

Let us recall the simplest types of singularities for smooth maps, before we compare them with the ones we have listed above. Let $y \in Int(X)$ be a point where $\text{rank}(d f_y) < 2$. The map $f\colon\, X\to \Sigma$ is said to have a \emph{fold singularity} at a point $y \in Int(X)$, if around $y$ and its image one can choose local coordinates so that the map is given by
\[(t, x_1, x_2, x_3) \mapsto (t, \pm x_1^2 \pm x_2^2 \pm x_3^2) ,\]
and a \emph{cusp singularity} if instead the map is locally given by
\[(t, x_1, x_2, x_3) \mapsto (t, x_1^3 + t x_1 \pm x_2^2 \pm x_3^2) . \]
We say that $y$ is a \emph{definite fold point} if all the coefficients of $x_i^2$, $i=1,2,3$ in the first local model is of the same sign. It is called an \emph{indefinite fold point} otherwise. From a special case of Thom's transversality theorem it follows that any smooth map from an $n$-dimensional manifold to a $2$-manifold, for $n \geq 2$, can be approximated arbitrarily well by a map with only fold and cusp singularities \cite{T, Lev0}. We will call a smooth map $f\colon\, X \to \Sigma$ with only fold and cusp singularities a \emph{generic map}. The singular set of $f$ is a finite set of circles, which are composed of finitely many cusp points and a finite collection of arcs and circles of fold singularities. From the very definitions we see that round singularities are circles of indefinite fold singularities.

As the set of generic maps is open and dense in $C^{\infty}(X, \Sigma)$ topologized with the $C^{\infty}$ topology, any broken Lefschetz fibration can be homotoped to a map with only fold and cusp singularities. In the next section, combining different topological modifications of singularities (\cite{Lev, S, Lek}), we will show how one can effectively eliminate the definite fold singularities and cusps so to homotope a generic map to a broken Lefschetz fibration, implying the abundance of broken Lefschetz fibrations.

\medskip
\section{Proofs of main theorems} \label{Sec:Existence}

\begin{proof} [Proof of Theorem \ref{BLFexistence}]
Given an arbitrary closed smooth oriented $4$-manifold $X$ and a square zero surface $F$ in it, we will now show how to construct a broken Lefschetz fibration $f:\, X \to S^2$ so that $F$ is a regular fiber. The surface $F$ is not required to be homologically essential, so one always has such an $F$.

We begin with constructing a surjective continuous map $g: X \to S^2$ which is a smooth fibration around $F$. This is the Thom-Pontrjagin construction. Let $N \cong D^2 \x F$ be a tubular neighborhood of $F$, and $pr_1: N \to D^2$ be the obvious projection. Let $r: D^2 \to S^2$ be the quotient map defined by collapsing $\partial D^2$ to a point. We can assume that $r( \partial D^2) = \{\text{NP}\}$ and $r(0) = \{\text{SP}\}$, where $\{\text{NP}\}$ and $\{\text{SP}\}$ are the north pole and the south pole in $S^2$, respectively. The composition $r \circ pr_1$ is a surjective map from $N$ to $S^2$, which we can extend all over $X$ by mapping $X \setminus N$ to $\{\text{NP}\}$. Hence we get a surjective continuous map $g: X \to S^2$, which is smooth away from $g^{-1}(\{\text{NP}\})$. Let $N_0 \subset \text{Int}(N)$ be the preimage of the southern hemisphere under $g$.

The map $g$ can be approximated arbitrarily well by a generic map $h: X \to S^2$ relative to $N_0$, where $g$ is already a submersion. However, this map may have definite fold points, which we will avoid by applying Saeki's algorithm given in \cite{S}.\footnote{I am grateful to Yasha Eliashberg for a conversation that led me to find out about Saeki's work.} For the convenience of the reader, we will give an outline of this algorithm here. Set $\hat{X} = X \,\#\, S^1 \x S^3 = (X \setminus D_1) \, \cup (S^1 \x S^3 \setminus D_2)$, where $D_1, D_2$ are $4$-disks, and let $\gamma = S^1 \x \{pt\} \subset S^1 \x S^3 \setminus D_2$. Take a continuous map $\hat{g}: \hat{X} \to S^2$ which is homotopic to $g$ on $X \setminus D_1$ by a homotopy supported around the connect sum region and is null-homotopic on $S^1 \x S^3 \setminus D_2$. Then approximate $\hat{g}$ by a generic map $\hat{h}: X \to S^2$ relative to $N_0$ instead. Saeki then shows how to homotope $\hat{h}$ to a new generic map which contains $\gamma$ in its singular set, and such that $\gamma$ contains an arc of definite fold points and exactly two cusp points which constitute a ``matching pair'' as described in \cite{Lev}. We further homotope the map so that all other singular circles containing definite fold points become circles of the same type as $\gamma$, and are bounded by disjoint disks which intersect with the singular set only along their boundaries. We then match one cusp point on any of these circles with a cusp point on $\gamma$, and reduce the singular set by one circle at a time. The joining curves used to perform this operation are found in abundance, and can be taken away from $N_0$. Finally, we homotope the map again to turn $\gamma$ into a singular curve with only definite fold points, so to apply a trick where one surgers out a tubular neighborhood $S^1 \x D^3$ of $\gamma$ and glues in a $D^2 \x S^2$ to which the map on the complement extends without any new singular locus. (This step of the proof does not work when the base is of higher genus; see \cite{S}.) We get $X \# S^4$, and a map $f$ which is homotopic to $\hat{h}$ on $X \setminus D_1$. A slight modification of this process, as Saeki argues, guarantees that $f$ is indeed homotopic to the very first map $g$ on the whole of $X$.

\begin{figure}[ht] \label{CuspModification}
\begin{center}
\includegraphics[scale=0.7]{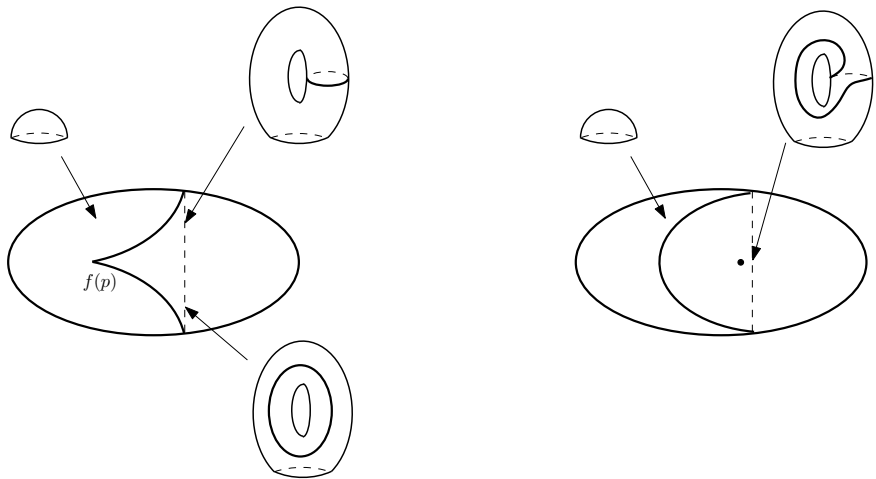}
\caption{\footnotesize ``Cusp modification''. On the left: The generic map $f$ from a small ball neighborhood of the cuspidal singularity to $D^2$. The two cycles that get collapsed across the singular set are depicted on the punctured torus fibers. On the right: The broken Lefschetz fibration over $D^2$ to replace the cusp model. The vanishing cycle of the Lefschetz singular point is as indicated.}
\end{center}
\end{figure}

Thus we observe that elimination of definite fold points can be performed away from $N_0$, and by keeping the homotopy in each step relative to $N_0$. The map $f: X \to S^2$ we get at the end is a smooth surjective map with only indefinite fold points and possibly with some cusp points. These cuspidal singularities appear as isolated points that bound arcs of \textit{indefinite} fold singularities, and therefore they can only acquire the index $3/2$ in the lack of definite fold points. In other words, we can choose local charts around them where $f$ is given by
\[(t, x_1, x_2, x_3) \mapsto (t, x_1^3 + t x_1 + x_2^2 - x_3^2). \]
This restricts the topology of preimages (`fibers') of $f$ around the cusp point to the case worked out by Lekili in \cite{Lek}. Let $B$ be a ball neighborhood of such a cusp point $p$ that intersects with the singular set of $f$ along an arc $C$ containing no other cusp points. The image of this arc under $f$, is a cusp curve in $D^2$. The restriction $f: B \to D^2$ has a simple topology where the preimages over one region of the base $D^2$ split by $f(C)$ are punctured tori, whereas the preimages over the other region are disks. (See Figure 1.) Observe that for $t = -1/2$ the map $f(-1/2, x_1, x_2, x_3)$ is a Morse function with two canceling critical points of indices one and two. Lekili describes a broken Lefschetz fibration over $D^2$ whose singular set consists of an arc of indefinite fold singularities and a Lefschetz singular point. (See Figure 1.) The attaching circle is carefully chosen so that not only the total space of this fibration is $D^4$ but also it matches with $f$ on the boundary. Hence, we can modify $f$ around $p$ by trading these two local models. Remarkably (thus, see Remark \ref{uniqueness}) one can perform this modification by a homotopy supported in $B$, as shown in \cite{Lek}.

Applying this modification to every cusp point we get a smooth surjective map $\bar{f}: X \to S^2$ with only Lefschetz and indefinite fold singularities; namely, a broken Lefschetz fibration. Since we have kept all the homotopies relative to $N_0$, then $\bar{f}^{-1}(\{\text{SP}\}) = g^{-1}(\{\text{SP}\}) = F$, so $F$ is a fiber.

The cardinality of the number of cusp points has the same parity as $\eu(X)$ due to a theorem of Thom, which can be reduced to zero or one after a homotopy as proved by Levine in \cite{Lev}. To prove the last part of our statement, we first apply Levine's elimination procedure to the map $\hat{h}$ above, while once again taking all the homotopies relative to $N_0$. This is possible since these homotopies can be supported around the singular circles and bands connecting them ---as in the case of Saeki's procedure. Afterward, we eliminate the definite fold points as before to get a generic map $f: X \to S^2$ with only one cusp point if $\eu(X)$ is odd, and none if $\eu(X)$ is even. (Note that the only new cusp points produced in Saeki's algorithm are eliminated right after along with the singular circles they are contained in.) Therefore we already have a broken fibration (with no Lefschetz singularities) when $\eu(X)$ is even, and otherwise we modify the only cusp point to get a Lefschetz critical point as above. This completes the proof of Theorem \ref{BLFexistence}.
\end{proof}

\begin{remark} \label{uniqueness} \rm 
The most natural question that follows is what the `uniqueness' theorem for broken Lefschetz fibrations would be (\cite{GK, P}). It would be interesting to see if a parametric version of Saeki's elimination procedure could be used to connect two homotopic broken Lefschetz fibrations $f_0, f_1 : X \to S^2$ by a one-parameter family of maps $f_t: X \to S^2$, $0 \leq t \leq 1$, which are broken Lefschetz fibrations except for finitely many values of $t$. One could then apply some basic moves to obtain one fibration from another in the same homotopy class, using the main result in Lekili \cite{Lek}. This is of great interest when considered with Perutz's program, where one calculates invariants associated to broken Lefschetz fibrations, which are conjectured to be invariants of the underlying smooth structure \cite{P}.
\end{remark}

Next we prove Theorem \ref{BLPexistence}:

\begin{proof}[Proof of Theorem \ref{BLPexistence} ]
Let $F$ be an orientable surface in $X$ with $[F]^2 = m > 0$. Let $\tilde{X}$ be the blow-up of $X$ at $m$ points on $F$, $E_1, \ldots, E_m$ be the exceptional spheres, and $\tilde{F}$ be the proper transform of $F$. Each $E_i$ intersects with $\tilde{F}$ positively at a point. Let $N_i$ be disjoint tubular neighborhoods of $E_i$ and $N_0$ be a tubular neighborhood of $\tilde{F}$ in $X$. Now set $g_0 : N_0 \cong D^2 \x \tilde{F} \to D^2$ to be the projection onto the first disk component, and $g_i : N_i \to S^2$ to be the radial projection onto $S^2$. Arranging the latter so that all the disk sections $E_i \cap N_0$ are mapped onto the southern hemisphere of $S^2$, we can define a surjective map $g_N$ from the closed set $N = \bigcup_{i=0}^m N_i$ to $S^2$. Note that $g_N$ is a submersion on $\text{Int}(N)$; preimages of the interior points of the southern hemisphere are copies of $\tilde{F}$, preimages of the interior points of the northern hemisphere are disks, and those of the equatorial points are copies of $\tilde{F}$ with $m$ disjoint open disks removed.

We need to extend $g_N$ to a continuous map $g: \tilde{X} \to S^2$. A way to do this is by first defining the map on a collar of the $\partial (X \setminus N)$ using $g|_{\partial N}$, and then mapping all the remaining points in the interior to $\{\text{NP}\}$. With such a $g$ in hand, we can run the proof of our previous theorem, mutadis mutandis to get a broken Lefschetz fibration $\bar{f}: \tilde{X} \to S^2$. The only difference is that at the end every fiber intersects each $E_i$ exactly at one point, so $E_i$ is a section, for all $i = 1, \ldots, m$. Thus we can blow-down $E_i$ to get the broken Lefschetz pencil on $X$ with $F$ as its fiber. The assertion on reducing the number of Lefschetz critical points to one and zero, respectively, depending on the parity of $\eu(X)$, is obtained in the same way as before.
\end{proof}

\begin{remark} \rm
In this article we restricted our attention solely on the existence of broken Lefschetz fibrations and pencils. Using advances in realizing prescribed singular loci, one can possibly improve our proof by making the zero locus of a given integral near-symplectic form coincide with the circles of indefinite folds, and with the same local models around each circle, as in \cite{ADK}. A theorem of Thom implies that if a $1$-manifold $Z$ is the singular locus of a generic map then its homology class in $\Z_2$ coefficients is nullhomologous, and Saeki has shown that this condition is sufficient to realize $Z$ as the singular locus of a generic map. Therefore the problem boils down to realizing certain $Z$, subject to Thom's homological condition, as the \textit{indefinite} singular locus of a generic map instead.
\end{remark}
 
\begin{remark} \label{ConnectedRoundLocus} \rm 
In \cite{Lek}, Lekili gives a new proof of reducing the number of near-symplectic circles, due to Perutz and Taubes independently \cite{P0, Ta}. Disguised in his proof is a simple algorithm to reduce the number of singular circles of a generic map with no definite singularities to one. Thus in Theorem \ref{BLFexistence} (resp.\,in Theorem \ref{BLPexistence}), one can instead choose the fibration (resp. the pencil) so that the round locus is connected in the expense of introducing new Lefschetz singularities. So one has the freedom to choose between having at most one Lefschetz singular point or one round singular curve. (The author is neutral about this.)
\end{remark}


\begin{remark} \label{ImmersedRoundImage} \rm
Round singular images of all the broken Lefschetz fibrations we constructed in the above proofs can be assumed to be immersions with double points. Saeki's procedure produces a stable map, which encompasses this property \cite{S}. The local modification around a cusp point preserves this property as well \cite{Lek}. However, Auroux-Donaldson-Katzarkov in \cite{ADK} and Gay-Kirby in \cite{GK} were able to arrange the round singular image of their broken Lefschetz pencils and achiral broken Lefschetz fibrations to have no double points, which we fail to reproduce here. In fact when one constructs broken Lefschetz fibrations using round handle additions, the round singular image is necessarily embedded. Gay and Kirby, who relied on round handles in their work, built this property into their definition of a broken Lefschetz fibration when conjecturing that not all closed $4$-manifolds support these fibrations \cite{GK}. Nevertheless, we will give below an example of a broken Lefschetz fibration on $\CP^2$ which meets the precise requirements of Gay and Kirby. Our construction is an incidence of a local modification around a negative node which was recently pointed out to the author by Lekili after a discussion on this very topic. It turns out that one can perturb a negative Lefschetz singularity locally to get a generic map with only indefinite fold and index $3/2$ cusp singularities, which yields a broken Lefschetz fibration after employing the cusp modification. Invoking the main theorem of \cite{GK}, one can then get broken Lefschetz fibrations on arbitrary $4$-manifolds with embedded round images. In short, there is no obstruction to having embeded round images either. (This is now contained in the new version of Lekili's preprint; see \cite{Lek}.)
\end{remark}

\begin{figure}[ht] \label{BLFonCP2}
\begin{center}
\includegraphics[scale=0.7]{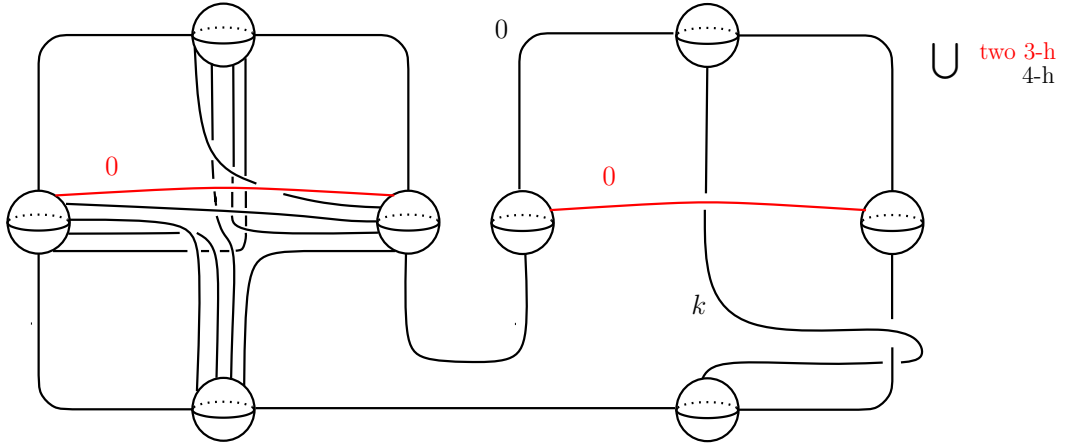}
\caption{\footnotesize A broken Lefschetz fibration on $\CP^2$.}
\end{center}
\end{figure}

\begin{example}  \rm [A broken Lefschetz fibration on $\CP^2$] \label{BLFexample}
We will construct an explicit broken Lefschetz fibration on $\CP^2$ using handle diagrams. For the descriptions used here the reader can turn to \cite{B}. Our example consists of a genus two Lefschetz fibration over a disk with three Lefschetz singularities (the higher side), a trivial sphere fibration over a disk (the lower side), and a round cobordism in between that consists of two round $2$-handles as shown in Figure 2. Importantly, there is a twisted gluing of the $T^3$ boundaries that appear in the round cobordism. (Observe that if we attach only one of the round $2$-handles to the higher side, and the other one dually as a round $1$-handle to the lower side, both pieces we get have $T^3$ boundaries which we can identify in many ways.) This manifests itself in how the $2$-handle from the lower side with arbitrary framing $k$ is included in our diagram. Lastly, recall that to use such a diagram one needs to verify that the indicated round $2$-handle attachments are legitimate. This is done by checking that the $2$-handles of the round $2$-handles are mapped onto themselves under the global monodromy prescribed by the sequence of Lefschetz handle attachments. This is a fairly simple task in our case, since all the vanishing cycles are collected on one genus. In fact, the monodromy splits, and on the nontrivial part it suffices to consider the product of Dehn twists on a torus given by $\tau_{a+b} \circ \tau_{2b-a} \circ \tau_{2a-b}$, where $a, b$ are standard generators of $\pi_1(T^2)$. This is comparable to the update in \cite{Lek}, and in particular the handle calculus for our example below contains an alternative description of the same modification around a negative \nolinebreak node. 

\begin{figure}[ht]  \label{BLFonCP2simplification}
\begin{center}
\includegraphics[scale=0.7]{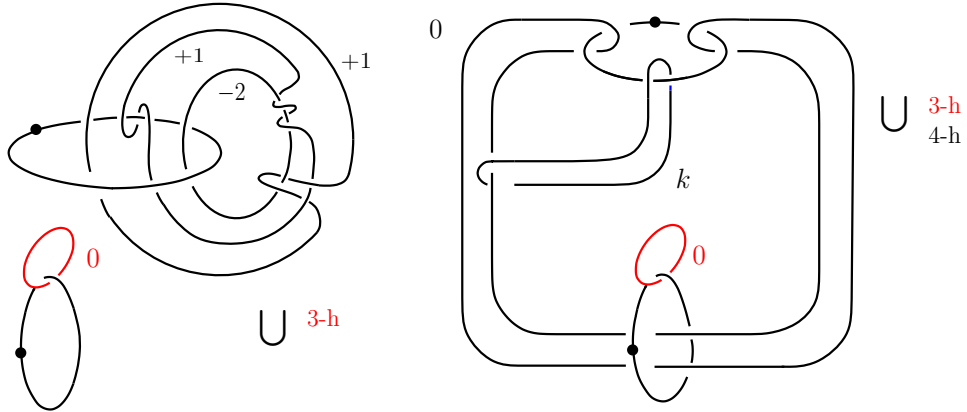}
\caption{\footnotesize Splitting of the handle diagram.}
\end{center}
\end{figure}

\enlargethispage*{\baselineskip}
The Figure 2 contains the complete diagram of our broken Lefschetz fibration where the $2$-handles corresponding to Lefschetz handles have fiber framing minus one. Using the $0$-framed $2$-handle of the round $2$-handle on the left, we can split the diagram as shown in Figure 3. Here we also switch to the dotted-circle notation to perform the rest of our handle slides and cancelations. All the handles in the subdiagram on the right, except for the $4$-handle of course, can be seen to cancel each other easily. This cancelation is identical to that of Example 3.5 given in \cite{B}. We can also remove the canceling $1$- and $2$- handle pair on the bottom left. It remains to simplify the rest of the diagram on the top left, where we have one $1$-handle, three $2$-handles, and a $3$-handle. First slide the $(+1)$-framed $2$-handles over the $(-2)$-framed $2$-handle to separate them from the $1$-handle, and cancel the $1$-handle against this $(-2)$-framed $2$-handle. After isotopies we obtain two $(+1)$-framed unknots, linking once. Now the obvious handle slide gives us two unknotted circles with framings $0$ and $+1$, respectively, where the former can be canceled against the $3$-handle. We are left with a $(+1)$-framed unknot and a $4$-handle, thus we obtain $\CP^2$.
\end{example}

\section{Getting connected fibers} \label{Sec:Connected}

We will start with a brief discussion of certain modifications of the round image which can be realized by homotopies of broken Lefschetz fibrations on the ambient $4$-manifold. These modifications are complementary to those studied in \cite{Lek}. Assume that the round image is immersed with double points, then it splits the base into several regions. The topology of fibers over any two \emph{adjacent} regions, i.e. regions that share a common arc boundary along $f(Z)$, are related by a fiberwise $2$-handle (and conversely by a $1$-handle) attachment. These fiberwise attachments relate a fiber $F_1$ to a fiber $F_2$ by removing an annulus from $F_1$ and capping off the two new boundary components by disks to get $F_2$. One can orient $Z$, and thus $f(Z)$, to indicate the direction of fiberwise $2$-handle attachments. Equivalently, we will do this by drawing small transverse arrows on the arcs of round image away from the double points, where the arrow will point in the direction of the $2$-handle attachments.

\begin{figure}[ht]  \label{R2inverse}
\begin{center}
\includegraphics[scale=0.7]{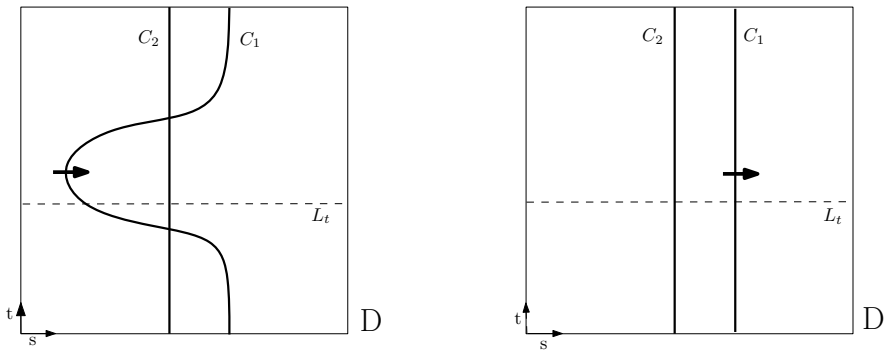}
\caption{\footnotesize A Reidemeister move for broken Lefschetz fibrations. The arrow on $C_2$ can point in either direction. }
\end{center}
\end{figure}

For what follows, we will need a specific move where one slides an arc $C_1$ in $f(Z)$ over another arc $C_2$ inside a disk $D \cong [0,1] \x [0,1]$, where $C_1$ and $C_2$ intersect at two points as shown in Figure 4. If we slice this picture by horizontal line segments $L_t(s)$, $s\in [0,1]$, we get a family of $3$-manifolds $M_t = f^{-1}(L_t)$, parametrized by $t \in [0,1]$. Each $M_t$ is determined by the regular fiber $f^{-1}(0,t)$ and the two handle additions over $C_1 \cap L_t$ and $C_2 \cap L_t$ as $s$ increases from $0$ to $1$. (The order of handle additions certainly depend on $t$ and for exactly two values of $t$ the critical points share the same image.) If the index of the handle attachment over $C_1 \cap L_t$ is greater than equal to that of the handle over $C_2 \cap L_t$, then this move in the base can be realized by a handle slide and therefore by an isotopy upstairs. This is the case for the slide depicted in Figure 4. So there is a family of isotopies parametrized by $t \in [0,1]$ determined by the slide in the $s$ direction, yielding an isotopy of $f^{-1}(D)$ which is identity close to $f^{-1}(\partial D)$. At each time of the isotopy we have a broken Lefschetz fibration. 

It is worth noting that similar arguments can be used to realize other moves by homotopies of broken Lefschetz fibrations. In this respect, the above move can be regarded as a Reidemeister II move, and it is not hard to verify that its `inverse' and Reidemeister III moves can also be used if we slide an arc in the direction of an arrow on it. Furthermore, one can slide the round image over a Lefschetz critical point provided that the arrow points in the sliding direction; this is ``pushing the Lefschetz critical point to the higher side'' described in \cite{B}. These observations all together allow one to slide an arc of the round image over several \textit{regions} in the direction of the associated arrow and get homotopic broken Lefschetz fibrations.

\begin{figure}[h] \label{FS}
\begin{center}
\includegraphics[scale=0.65]{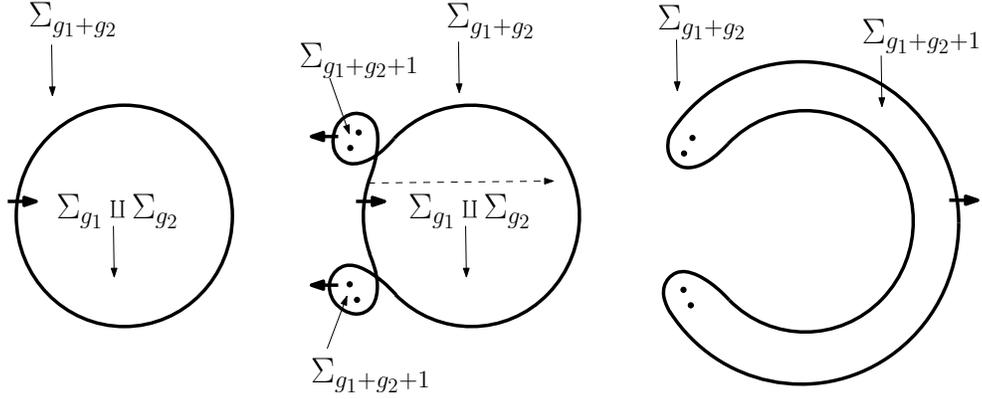}
\vspace{0.5cm}
\caption{\footnotesize ``Flip and slip''. On far left: The disk neighborhood of the round image of the original broken Lefschetz fibration with disconnected fibers in the interior. After two flips, we get the fibration in the middle where there are four new Lefschetz singularities and two double points. The direction of the slide to follow is indicated by the dashed arrow. On far right: The final fibration with connected fibers after the slip.}
\end{center}
\end{figure}

\begin{proof}[Proof of Proposition \ref{connected}] The proof will follow from a sequence of moves; called \emph{flip and slip} herein. By assumption $Z$ is connected and $f(Z)$ is embedded. Let $\Sigma_g$ denote a closed orientable surface of genus $g$. In Figure 5 we see a disk neighborhood of the \textit{lower side} in the base, where over the inner region split by the round image we have disconnected fibers $\Sigma_{g_1} \coprod \Sigma_{g_2} $ and over the outer part we have connected fibers $\Sigma{\,g_1 +g_2}$. Without loss of generality, we can leave the Lefschetz critical points out of this picture, since they could be pushed deep into the higher side after homotoping the original fibration. We first apply two simultaneous \emph{flips} as shown in the middle of Figure 5. This move is due to Auroux, and it can be described in two steps: One first applies the well-known flip move for generic maps, while getting two new cusp singularities and a new double point, and then brings in the cusp modification to trade the cusps with Lefschetz critical points. Since the latter can also be realized by a homotopy, we conclude that each flip upstairs gives a homotopic broken Lefschetz fibration. (See \cite{Lek} for the details.) We then make the big slide shown in Figure 5; the \emph{slip}. Using the guidance of the arrows, we can see that the topology of fibers over each region is as shown. The key point is that, when one passes from one region to another in the opposite direction of an arrow, that corresponds to a fiberwise $1$-handle attachment, which can not disconnect any fiber component. Along these homotopies the only time we fail to have a broken Lefschetz fibration is the first step of the two flips, where one gets two pairs of cuspidal singularities which are immediately smoothened out thereafter. 
\end{proof}

\begin{proof}[Proof of Corollary \ref{Always} ]
As noted in the introduction, we are granted with a broken Lefschetz pencil by \cite{ADK} whose blow-up gives a near-symplectic broken Lefschetz fibration with connected round locus and embedded round image. All we need to verify is that after applying Proposition \ref{connected} we get a near-symplectic fibration. It is clear that the involved modifications take place away from a regular neighborhood $N_0$ of a connected fiber $F$ of the initial broken Lefschetz fibration. Moreover, there is a near-symplectic form $\omega$ that evaluates positively on $F$. As we keep $N_0$ intact, the fibration with connected fibers we get at the end has $F$ as a fiber, too. Thus $\omega$ evaluates positively on all fibers of the resulting fibration as well, which allows us to apply the converse result of Auroux, Donaldson and Katzarkov \cite{ADK} in order to make the fibration near-symplectic. Throughout these homotopies the sections are preserved, so we can symplectically blow-down the exceptional spheres to obtain a broken Lefschetz pencil with the desired properties on the original $4$-manifold.
\end{proof}

\enlargethispage*{\baselineskip}
\begin{remark} \rm
We can in fact conclude more by observing that the modifications we use do not change the deformation class of the near-symplectic form. This is evident for the slip move, and it was argued for flips in \cite{Lek}. Also note that we invoked the existence result of \cite{ADK} so to produce a broken Lefschetz pencil that meets precisely with Perutz's requirements. If the round image is not asked to be embedded, then one can presumably use Theorem \ref{BLPexistence} instead along with the methods of this section for a more elementary proof.
\end{remark}


\begin{thebibliography}{99999}

\bibitem{ADK} D. Auroux, S.\,K. Donaldson and L. Katzarkov, {\em Singular Lefschetz pencils}, Geom. Topol. {\bf 9} (2005), 1043--114.

\bibitem{B} R.\,I. Baykur, {\em Topology of broken Lefschetz fibrations and near-symplectic $4$-manifolds,} preprint, arXiv:0801.0192.

\bibitem{GK} D.\,T. Gay and R. Kirby, {\em Constructing Lefschetz-type fibrations on four-manifolds,}  Geom. Topol. {\bf 11} (2007), 2075–-2115.

\bibitem{Lek} Y. Lekili, {\em Wrinkled Fibrations on Near-Symplectic Manifolds,} preprint, arXiv:0712.2202.

\bibitem{Lev0} H.\,I. Levine, {\em Singularities of differentiable mappings,} Proceedings of Liverpool
Singularities, Symposium I, Edited by C.T.C. Wall, Lecture Notes in Math. {\bf 192}, Springer (1971), 1--89.

\bibitem{Lev} H.\,I. Levine, {\em Elimination of cusps,} Topology {\bf 3}, no. 2 (1965), 263--296.

\bibitem{P0} T. Perutz, {\em Zero-sets of near-symplectic forms,} J. of Symp. Geom. {\bf 4} (2006), no. 3, 237--257.

\bibitem{P} T. Perutz, {\em Lagrangian matching invariants for fibred four-manifolds: I}, Geom. Topol. {\bf 11} (2007), 759--828.

\bibitem{S} O. Saeki, {\em Elimination of definite fold,} Kyushu J. Math. {\bf 60} (2006), 363--382.

\bibitem{Ta} C.\,H. Taubes, {\em A proof of a theorem of Luttinger and Simpson about the number of vanishing circles of a near-symplectic form on a 4-dimensional manifold,} Math. Res. Lett. 13 (2006), no. 4, 557--570.

\bibitem{T} R. Thom, {\em Les singularit\'es des applications diff\'erentiables,} Ann. Inst. Fourier, Grenoble {\bf 6} (1955--1956), 43--87.
\end{thebibliography}
\end{document}